\documentclass[12pt,a4paper]{amsart}

\usepackage{amscd}
\usepackage{amsmath}
\usepackage{amssymb}
\usepackage{amsthm}
\usepackage{color}
\usepackage{graphicx}
\usepackage{hyperref}
\usepackage[noabbrev,capitalize]{cleveref}
\usepackage{mathrsfs}
\usepackage{mathtools}
\usepackage{pict2e}
\usepackage{stackrel}
\usepackage[T1]{fontenc}
\usepackage{tikz-cd}
\usepackage[utf8]{inputenc}
\usepackage{wasysym}
\usepackage[all]{xy}
\usepackage{xcolor}
\usepackage{a4wide}

\usepackage{tikz}
\usetikzlibrary{shapes,fit,positioning,calc,matrix}
\tikzset{
  optree/.style={scale=.5,thick,grow'=up,level distance=10mm,inner sep=1pt},
  comp/.style={draw=none,circle,fill,line width=0,inner sep=0pt},
  dot/.style={draw,circle,fill,inner sep=0pt,minimum width=3pt},
  circ/.style={draw,circle,inner sep=1pt,minimum width=4mm},
  emptycirc/.style={draw,circle,inner sep=1pt,minimum width=2mm},  
  root/.style={level distance=10mm,inner sep=1pt},
  leaf/.style={draw=none,circle,fill,line width=0,inner sep=0pt},
  nodot/.style={draw,circle,inner sep=1pt},
}

\definecolor{Chocolat}{rgb}{0.36, 0.2, 0.09}
\definecolor{BleuTresFonce}{rgb}{0.215, 0.215, 0.36}
\hypersetup{colorlinks, citecolor=BleuTresFonce, linkcolor=Chocolat}

\usepackage[lf]{Baskervaldx} 
\usepackage[bigdelims,vvarbb]{newtxmath} 

\allowdisplaybreaks

\theoremstyle{plain}

\newtheorem{lemma}{Lemma}[section]
\newtheorem{theorem}[lemma]{Theorem}
\newtheorem{corollary}[lemma]{Corollary}
\newtheorem{proposition}[lemma]{Proposition}

\theoremstyle{definition}

\newtheorem{definition}[lemma]{Definition}
\newtheorem{remark}[lemma]{\sc Remark}

\DeclareMathAlphabet{\pazocal}{OMS}{zplm}{m}{n}

\def\calA{\pazocal{A}}
\def\calB{\pazocal{B}}
\def\calC{\pazocal{C}}

\def\calF{\pazocal{F}}

\def\calO{\pazocal{O}}
\def\calP{\pazocal{P}}
\def\calQ{\pazocal{Q}}
\def\calR{\pazocal{R}}

\def\calX{\pazocal{X}}

\DeclareMathOperator{\Lie}{Lie}
\DeclareMathOperator{\Com}{Com}
\DeclareMathOperator{\Zinb}{Zinb}

\DeclareMathOperator{\Ass}{Ass}
\DeclareMathOperator{\As}{As}
\DeclareMathOperator{\PL}{PreLie}

\DeclareMathOperator{\RB}{RB}
\DeclareMathOperator{\RF}{RatFct}
\DeclareMathOperator{\Mou}{Mould}

\DeclareMathOperator{\ad}{ad}

\def\cd{\mathsf{d}}
\def\cE{\mathsf{E}}
\def\cV{\mathsf{V}}

\DeclareMathOperator{\id}{id}
\DeclareMathOperator{\Hom}{Hom}
\DeclareMathOperator{\End}{End}

\DeclareMathAlphabet{\mathbbold}{U}{bbold}{m}{n}

\def\k{\mathbbold{k}}

\usepackage[backend=bibtex,
    sorting=anyt,
    isbn=false,
    url=false,
    doi=false,
    maxbibnames=100
    ]{biblatex}
\author{Vladimir Dotsenko}
\address{Institut de Recherche Math\'ematique Avanc\'ee, 
Universit\'e de Strasbourg et CNRS, 
7 rue Ren\'e-Descartes, 
67084 Strasbourg Cedex, 
France}
\email{vdotsenko@unistra.fr}

\author{Sergey Shadrin}
\address{Korteweg-de Vries Institute for Mathematics, University of Amsterdam, P. O. Box 94248, 1090 GE Amsterdam, The Netherlands}
\email{s.shadrin@uva.nl}

\title[Hidden structures behind ambient symmetries of the Maurer--Cartan equation]{Hidden structures behind ambient symmetries\\ of the Maurer--Cartan equation}

\begin{document}

\setcounter{tocdepth}{1}

\begin{abstract} 
For every differential graded Lie algebra $\mathfrak{g}$ one can define two different group actions on the Maurer--Cartan elements: the ubiquitous gauge action and the action of $\Lie_\infty$-isotopies of $\mathfrak{g}$, which we call the ambient action. In this note, we explain how the assertion of gauge triviality of a homologically trivial ambient action relates to the calculus of dendriform, Zinbiel, and Rota--Baxter algebras, and to Eulerian idempotents. In particular, we exhibit new relationships between these algebraic structures and the operad of rational functions defined by Loday. 
\end{abstract}

\maketitle

\section*{Introduction}

The goal of this paper is to exhibit some algebraic and combinatorial phenomena that we observed when thinking about a topic in homotopical algebra, namely the action of $\Lie_\infty$-isotopies of a complete differential graded (dg) Lie algebra on the set of Maurer--Cartan elements of that algebra. Specifically, in our joint work with Vaintrob and Vallette~\cite{MR4726567}, we needed the fact that actions of two such isotopies that are in a sense equivalent (that is, are exponentials of homologous $\Lie_\infty$-derivations) give Maurer--Cartan elements that are equivalent in the usual sense, that is related by the gauge action. The proof of that result in~\cite[Prop.~3.17]{MR4726567}, essentially going back to \cite[Lemma 3]{MR3335859}, uses a general homotopy-theoretic argument; in this paper, we shall discuss this result from other viewpoints, leading, in particular, to an explicit formula of the gauge trivialization of the homologically trivial ambient action.\par\smallskip

We begin by making this short summary more precise for non-experts.

\subsection*{Maurer--Cartan elements and their gauge symmetries}
The Maurer--Cartan equation
 \[
d\alpha+\frac12[\alpha,\alpha]=0
 \]
in dg Lie algebras is arguably one of the most important equations in mathematics. Originating in differential geometry and very prominent in mathematical physics, in the past decades it has been ubiquitous in deformation theory: according to the general philosophy that goes back to Deligne \cite{Deligne93}, Drinfeld \cite{MR3285856} and others and was made by Lurie \cite{Lur} and Pridham \cite{MR2628795},  for each reasonable structure on an algebraic or a geometric object, there exists a dg Lie algebra whose Maurer--Cartan elements are in a bijection with different choices of the structure, and which can be used to study the question of deformations of a given structure and the question of equivalence of different structures. In particular, the \emph{gauge equivalence} of Maurer--Cartan elements in complete dg Lie algebras plays an important role. A clear way to define that equivalence is via the so called ``differential trick''. Namely, we shall assume that the differential of $\mathfrak{g}$ is internal, that is, that there exists an element~$\delta$ with $[\delta,\delta]=0$ satisfying $[\delta,x]=d(x)$ for all $x\in\mathfrak{g}$. (It is possible to show that such an element can be adjoined in a consistent way, so without loss of generality we may assume its existence.) In that case, the Maurer--Cartan equation may be written as $[\delta+\alpha,\delta+\alpha]=0$, and under the completeness assumption, for each $x\in\mathfrak{g}_0$, and each Maurer--Cartan element~$\alpha$ of $\mathfrak{g}$, the formula $\exp(\ad_x)(\delta+\alpha)$ defines another Maurer--Cartan element, which is said to obtained from $\alpha$ by a gauge symmetry action of $x$.

\subsection*{Maurer--Cartan elements and ambient symmetries}
Of course, a differential graded Lie algebra itself is also an instance of a structure, and so this philosophy is, in a sense, applicable to itself. Precisely, if $\mathfrak{g}$ is a dg Lie algebra, we can encode its Lie bracket by a Maurer--Cartan element $\phi$ of the dg Lie algebra associated to the pre-Lie algebra
 \[
\mathfrak{b}:=\left(\prod_{n\ge 2}\End_{s\mathfrak{g}}(n)^{S_n},\partial,\star\right).
 \]
Generally, Maurer--Cartan elements of that dg Lie algebra encode $\Lie_\infty$-algebra structures on the underlying chain complex of $\mathfrak{g}$; restricting ourselves to dg Lie algebra structures amounts to considering elements supported at arity $2$. We may consider the twisted chain complex $\mathfrak{b}^\phi$ with the differential $\partial+[\phi,-]$. Cycles of degree zero in that complex are precisely the \emph{$\Lie_\infty$-derivations} of the Lie algebra $\mathfrak{g}$ with zero component of arity one. The associated  Lie algebra of $\mathfrak{b}$ is complete with respect to the arity filtration, and one can define the exponential map; the exponentials of the $\Lie_\infty$-derivations of $\mathfrak{g}$ of the form described above are therefore well defined and are precisely the \emph{$\Lie_\infty$-isotopies} of $\mathfrak{g}$, that is the $\Lie_\infty$-automorphisms whose arity one part is the identity map. It is possible to show that for every cycle $\lambda$ of degree zero in $\mathfrak{b}^\phi$ and for every Maurer--Cartan element $\alpha$ of $\mathfrak{g}$, the formula
 \[
\lambda.\alpha:=\sum_{n\ge 1}\frac{1}{n!}(e^\lambda)_n(\alpha,\ldots,\alpha)
 \]  
defines another Maurer--Cartan element; thus, the group of $\Lie_\infty$-isotopies of $\mathfrak{g}$ acts on Maurer--Cartan elements of $\mathfrak{g}$. We call this action the \emph{ambient action}; it is defined in a very different way from the gauge action. 

This paper emerged from contemplating the above-mentioned result needed in \cite{MR4726567} stating that if $\lambda$ is a degree zero cycle in the twisted chain complex $\mathfrak{b}^\phi$ that happens to be trivial in the homology, then its action on Maurer--Cartan elements is gauge trivial, that is, for every Maurer--Cartan element $\alpha$, the two Maurer--Cartan elements $\alpha$ and $\lambda.\alpha$ are gauge equivalent. (Strictly speaking, the context of \cite{MR4726567} requires one to work with shifted Lie algebras, but we choose to ``unshift'' them for the purpose of this paper to make it more readable for the audience of experts in combinatorial algebra.) While working on the proof of this result, we found that several well known combinatorial algebraic structures emerged along the way, and we decided to spell out the corresponding results, hoping that both the homotopy theorists and the experts in combinatorial algebra might find them useful elsewhere. The structures that are the most crucial for furnishing explicit formulas are dendriform algebras, Zinbiel algebras, and Rota--Baxter algebras. A dendriform algebra is a vector space equipped with two binary operations $\prec$ and $\succ$ satisfying the identities
\begin{gather*}
(a_1\prec a_2)\prec a_3=a_1\prec(a_2a_3),\\
(a_1\succ a_2)\prec a_3=a_1\succ(a_2\prec a_3),\\
(a_1a_2)\succ a_3=a_1\succ(a_2\succ a_3),
\end{gather*}
where $a_1a_2=a_1\prec a_2+a_1\succ a_2$; this notion was originally introduced by Loday \cite{MR1860994} with a motivation from algebraic K-theory and studied quite extensively since. A dendriform algebra in which $a_1\prec a_2=a_2\succ a_1$ is called a Zinbiel algebra; concretely, it is an algebra with one binary operation $\nu$ satisfying the identity
 \[
\nu(a_1,\nu(a_2,a_3))=\nu(\nu(a_1,a_2),a_3)+\nu(\nu(a_2,a_1),a_3).
 \]
This notion emerged in work of Cuvier \cite{MR1258404} on Leibniz cohomology, which in particular implies that the operad of Zinbiel algebras is the Koszul dual of the operad of Leibniz algebras (this duality is forever commemorated in the term ``Zinbiel'' proposed by Lemaire). Finally, a Rota--Baxter algebra is a vector space equipped with a commutative associative product $\cdot$ and an endomorphism $R$ such that 
 \[
R(a_1)\cdot R(a_2)=R(R(a_1)\cdot a_2+a_1\cdot R(a_2)).
 \]
This notion goes back to work of Baxter \cite{MR119224} that was widely popularized by Rota \cite{MR244070}; nowadays Rota--Baxter operators are considered for many different algebraic structures, not just the commutative product. We refer the reader to \cite{MR2357335} for further references concerning these structures.  

\subsection*{Structure of the paper.} In Section \ref{sec:reformulate}, we reformulate the equations describing homologically trivial ambient symmetries in a particularly suggestive way that motivates the rest of the paper. In Section \ref{sec:operadic}, we record some results in operad theory that will be useful for us; some of them, in particular, the general ``Foissy functor'' (generalizing a pre-Lie algebra construction of Foissy \cite{FoissyNotes}) and its conceptual interpretation, are, to the best of our knowledge, new. In Section \ref{sec:euler}, we use calculations in the operad of rational functions of Loday \cite{MR2676956} to give an explicit formula for the gauge trivialization on the Lie algebra level. In Section \ref{sec:dendriform}, we use calculations in free dendriform algebras to give an explicit formula for the gauge trivialization on the Lie algebra level; we then show that the formulas thus obtained are equivalent. Finally, in Appendix \ref{app:RB}, we show that the operad of Rota--Baxter commutative algebras embeds into the operad of rational functions as a suboperad.

\subsection*{Conventions. } All vector spaces in this paper are defined over a field $\k$ of characteristic zero. All chain complexes are graded homologically, with the differential of degree~$-1$. To handle suspensions of chain complexes, we introduce a formal symbol~$s$ of degree~$1$, and define, for a graded vector space~$V$, its suspension $sV$ as $\k s\otimes V$. We assume certain fluency in the language of operads and combinatorial species; the reader who requires assistance with these is referred to the monographs \cite{MR2954392} and \cite{MR1629341} respectively. For details concerning Maurer--Cartan elements, gauge symmetries, and the related pre-Lie algebra calculus, the reader is invited to consult \cite{DSV16,MR4621635}.

\subsection*{Acknowledgements.} We thank Arkady Vaintrob and Bruno Vallette for discussions of this work while working on \cite{MR4726567}. The first author was supported by Institut Universitaire de France and by the French national research agency [grant ANR-20-CE40-0016]. The second author was supported by the Netherlands Organization for Scientific Research. The first draft of this paper was completed while the first author was visiting the Banach center in Warsaw during the Simons semester ``Knots, homologies, and physics''; he thanks this institution for hospitality and excellent working conditions. He also thanks Loïc Foissy for comments on the origin of what we call ``the Foissy functor'', and the participants of the conference ``70 years of Magnus series'', especially Ruggero Bandiera, Kurusch Ebrahimi--Fard, Dominique Manchon, and Frédéric Patras for comments and useful discussions.

\section{A different formulation of homologically trivial ambient symmetries}\label{sec:reformulate}

We begin with re-organizing the formulas describing ambient symmetries in the case that we are considering. Let us introduce a larger pre-Lie algebra which will be useful to place all elements we consider on the same ground. As hinted at in the introduction, we shall first extend $\mathfrak{g}$ to be able to use the ``differential trick'' \cite[Sec.~1]{DSV16}; this amounts to considering the graded vector space 
 \[
\overline{\mathfrak{g}}=\mathfrak{g}\oplus \k\delta
 \]
where the homological degree of $\delta$ is equal to $-1$, $[\delta,\delta]=0$, and $[\delta,x]=d(x)$ for all $x\in \mathfrak{g}$. We then consider the pre-Lie algebra
 \[
\overline{\mathfrak{b}}:=\prod_{n\ge 0}\End_{s\overline{\mathfrak{g}}}(n)^{S_n},\partial,\star).
 \]
It contains $\mathfrak{b}$ as a subalgebra, but also contains $s\mathfrak{g}$ as a subspace 
 \[
\End_{s\mathfrak{g}}(0)^{S_0}\subset\End_{s\overline{\mathfrak{g}}}(0)^{S_0}.
 \]
 We shall denote the element of $\overline{\mathfrak{b}}$ corresponding to the homological shift $s\kappa$ for $\kappa\in\mathfrak{g}$ by $\tilde{\kappa}$; conversely, for an element $\rho\in\overline{\mathfrak{b}}$, we shall denote the inverse homological shift of its arity zero part by $\overline{\rho}$.

For a Maurer--Cartan element $\alpha$ of $\mathfrak{g}$, the action $\lambda\cdot\alpha$ can be explicitly written in the pre-Lie algebra $\overline{\mathfrak{b}}$ as
\begin{equation}\label{eq:Gauge1}
\widetilde{\lambda\cdot\alpha}=\overline{e^\lambda \circledcirc \left(\mathbb{1}+\widetilde{\alpha}\right)} =
\widetilde{\alpha}+\sum_{i\ge 1}\frac{1}{i!} \overline{r_\lambda^{i-1} (\lambda) \circledcirc (1+ \widetilde{\alpha})} 
\,.
\end{equation}

Suppose that $\lambda\in \mathfrak{b}^\phi$ is a cycle of degree zero that is homologous to zero, so that there exists $f \in \mathfrak{b}$ of degree one satisfying $\partial^{\varphi}(f)=\lambda$. We wish to show that the action of $\lambda$ on Maurer--Cartan elements of $\overline{\mathfrak{g}}$ is gauge trivial, in other words, using the differential trick, that for any Maurer--Cartan element $\alpha$ of $\mathfrak{g}$ there exists an element $\xi\in\mathfrak{g}_0$ such that
\begin{equation}\label{eq:GaugeMain}
\exp(\ad_\xi)(\delta+\alpha) = \delta+\lambda.\alpha 
\end{equation}
in the Lie algebra $\overline{\mathfrak{g}}$. Motivated by Formula \eqref{eq:Gauge1}, let us introduce the elements $\lambda_i, f_i\in \mathfrak{g}$ by the formulas
\begin{align*}
\widetilde{\lambda_i} := \frac{1}{(i-1)!} \overline{r_{\lambda}^{i-1} (\lambda) \circledcirc (1+ \widetilde{\alpha})}; \\
\widetilde{f_i} := \frac{1}{(i-1)!} \overline{r_{\lambda}^{i-1} (f) \circledcirc (1+ \widetilde{\alpha})}.
\end{align*}

To work with these formulas, we define an auxiliary weight grading $w$, putting $w(\lambda)=w(f)=1$ and $w(\widetilde{\alpha})=w(\widetilde{\delta})=0$, so that $w(\lambda_i)=w(f_i)=i$; this will allow us to collect and separate elements with the same weight grading.
Note that by \cite[Prop.~3]{DSV16}, we have
\begin{gather*}
\sum_{i\ge 1}\widetilde{\lambda_i}=\overline{\sum_{i\ge 1}\frac{1}{(i-1)!} r_{\lambda}^{i-1} (\lambda) \circledcirc (1+ \widetilde{\alpha})}=
\overline{\lambda\circledcirc e^{\lambda}\circledcirc(1+ \widetilde{\alpha})}=\overline{\lambda\circledcirc (1+\widetilde{\lambda.\alpha})},\\
\sum_{i\ge 1} \widetilde{f_i}=\overline{\sum_{i\ge 1}\frac{1}{(i-1)!} r_{\lambda}^{i-1} (f) \circledcirc (1+ \widetilde{\alpha})}=
\overline{f\circledcirc e^{\lambda}\circledcirc(1+ \widetilde{\alpha})}=\overline{f\circledcirc (1+\widetilde{\lambda.\alpha})} .
\end{gather*}
The first of these equations implies that 
 \[
\sum_{n\ge 1}\widetilde{\lambda_i}=\overline{\partial^\phi(f)\circledcirc (1+\widetilde{\lambda.\alpha})}=\\
\overline{(\delta\star f+\phi\star f-f\star\phi-f\star\delta)\circledcirc (1+\widetilde{\lambda.\alpha})}.
 \]
Since
 \[
(\phi\star f)\circledcirc (1+\widetilde{\lambda.\alpha})=\left[\overline{f\circledcirc (1+\widetilde{\lambda.\alpha})},\lambda.\alpha\right]
 \]
and 
 \[
[\delta,\lambda.\alpha]+\frac{1}{2}[\lambda.\alpha,\lambda.\alpha]=0,
 \] 
this implies 
 \[
\sum_{n\ge 1}\lambda_i=\left[\sum_{i\ge 1} f_i,\delta+\lambda\cdot\alpha\right]=\left[\sum_{i\ge 1} f_i,\delta+\alpha+\sum_{j\ge 1}\frac{1}{j}\lambda_j\right].
 \]
Separating elements according to their weight gradings, we have 
\begin{equation} \label{eq:DifferentialOnf-i}
[f_n,\delta+\alpha] = 
\lambda_n - \sum_{\substack{i+j = n \\ i,j\geq 1}} \frac{1}{j}[f_i,\lambda_j] \,.
\end{equation}

Examining this formula, we see that the key algebraic structure used in it is \emph{not} the Lie bracket $[-,-]$ but its weighted version. In fact, this structure can be found in the literature. In \cite{FoissyNotes}, Foissy proposed a beautiful construction of a pre-Lie algebra structure on any Lie algebra $L$ with a strictly positive weight grading, $L=\bigoplus_{n>0}L_n$; that structure is given by
 \[
x\triangleleft y:=\frac{|x|}{|x|+|y|}[x,y].
 \]
It was noticed by Al-Kaabi, Manchon and Patras in \cite{MR3818652} that if conjugated by the grading operator (the operator that multiplies $L_n$ by $n$; since the grading is strictly positive, this operator is invertible), this structure becomes
 \[
x\blacktriangleleft y:=\frac{1}{|y|}[x,y],
 \]
which is precisely what we find in Formula \eqref{eq:DifferentialOnf-i}.  

\section{Operadic constructions}\label{sec:operadic}

Let us fix an operad $\calP$ generated by binary operations. Below we shall refer to the Manin products $\medbullet$ and $\medcirc$ defined by Ginzburg and Kapranov \cite{MR1301191,MR1360619} (see also \cite{MR2427978}) that generalize the corresponding operations for quadratic algebras defined by Manin \cite{MR1016381}.

\subsection{Disuccessors and Rota--Baxter operators}

Recall that in \cite{MR3021790}, the notion of a disuccessor $\mathrm{DSu}(\calP)$ was proposed. This operad is generated by splittings $\omega_\prec$ and $\omega_\succ$ of generators $\omega$ of $\calP$, and these splittings satisfy some explicitly given identities, see \cite[Def.~2.19]{MR3021790}. The prototypical examples of a disuccessor are given by the operad of dendriform algebras, which is the disuccessor of the operad of associative algebras, the operad of pre-Lie algebras which is the disuccessor of the operad of Lie algebras, and the operad of Zinbiel algebras, which is the disuccessor of the operad of commutative associative algebras. If an operad is quadratic, the disuccessor of $\calP$ coincides with the Manin black product $\calP\medbullet\PL$, but the assumption on quadraticity is not necessary. In fact, this notion can also be guessed from the results of Schedler \cite{MR3203374}: generalizing Proposition 3.3.3 of \emph{loc. cit.} in the following way, as a direct calculation shows.

\begin{proposition}
A structure of a $\mathrm{DSu}(\calP)$-algebra on a vector space $A$ is equivalent to a structure of a $\k[t]$-linear twisted $\calP$-algebra on the free $\k[t]$-module generated by a $A$, where both $t$ and $A$ are placed in arity one. The equivalence is given by 
 \[
\omega(a_1,a_2)=t(1)\left(\omega_\prec(a_1,a_2)\right)(2)-t(2)\left(\omega_\succ(a_2,a_1)\right)(1).
 \]
for all generators $\omega$ of $\calP$ and all $a_1,a_2\in A$. Here, for a linear species $\calX$, we denote by $x(i)$ the element corresponding to $x$ in $\calX(\left\{i\right\})$.
\end{proposition}

Recall that a Rota--Baxter operator on a $\calP$-algebra $A$ is a map $R\colon A\to A$ satisfying
 \[
\omega(R(a_1),R(a_2))=R(\omega(R(a_1),a_2)+\omega(a_1,R(a_2)))
 \]
for all generators $\omega$ of $\calP$ and all $a_1,a_2\in A$. We shall denote the operad of $\calP$-algebras with a Rota--Baxter operator by $\RB(\calP)$.

It is established in \cite[Th.~4.4]{MR3021790} that there exists a morphism of operads 
 \[
\mathrm{DSu}(\calP)\to\RB(\calP)
 \]
given by 
 \[
\omega_\prec\mapsto \omega\circ_2 R,\quad
\omega_\succ\mapsto \omega\circ_1 R.
 \]
Equivalently, for every Rota--Baxter operator $R$ on a $\calP$-algebra $A$, the operations 
 \[
\omega_\prec(a_1,a_2)=\omega(a_1,R(a_2)),\quad
\omega_\succ(a_1,a_2)=\omega(R(a_1),a_2)
 \]
define on $A$ a structure of an $\mathrm{DSu}(\calP)$-algebra (for $\calP$ one of the operads $\Com$, $\Lie$, $\Ass$, this goes back to \cite[Prop.~5.1]{MR1846958}, and for general binary quadratic operads it is proved in \cite[Th.~4.2]{MR2661522}). We complement this with the following result which is proved by a straightforward computation.

\begin{proposition}
There is a morphism of operads
 \[
\RB(\calP)\to\calP\underset{\mathrm{H}}{\otimes}\RB(\Com)
 \]
sending each operation $\omega\in\calP$ to $\omega\otimes\mu$, and $R$ to $\id\otimes R$ (here $\mu$ is the generator of $\Com$).
\end{proposition}

As a consequence, we have the following result featuring the operad of Zinbiel algebras $\Zinb=\mathrm{DSu}(\Com)$.

\begin{corollary}\label{cor:RBC-P}
We have a composite map 
 \[
\mathrm{DSu}(\calP)\to\RB(\calP)\to \calP\underset{\mathrm{H}}{\otimes}\RB(\Com),
 \]
implying that the tensor product of a $\calP$-algebra and a $\RB(\Com)$-algebra is equipped with a canonical $\mathrm{DSu}(\calP)$-algebra structure. Moreover, there is a commutative diagram 
 \[
 \xymatrix@M=6pt{
\mathrm{DSu}(\calP)\ar@{->}^{}[rr] \ar@{->}_{}[d] & & \calP\underset{\mathrm{H}}{\otimes}\Zinb \ar@{->}^{}[d]  \\ 
\RB(\calP)\ar@{->}_{}[rr]  & & \calP\underset{\mathrm{H}}{\otimes}\RB(\Com)      
 } 
 \]
\end{corollary}

\begin{proof}
The map $\mathrm{DSu}(\calP)\to \calP\underset{\mathrm{H}}{\otimes}\RB(\Com)$ is the composite map 
 \[
\mathrm{DSu}(\calP)\to\RB(\calP)\to \calP\underset{\mathrm{H}}{\otimes}\RB(\Com).
 \]
Moreover, the definitions of all the maps involved immediately imply that this map factors through 
$\calP\underset{\mathrm{H}}{\otimes}\mathrm{DSu}(\Com)\cong \calP\underset{\mathrm{H}}{\otimes}\Zinb$. 
\end{proof}

Let us remark that, if $\calP$ is a quadratic operad, the map $\mathrm{DSu}(\calP)\to \calP\underset{\mathrm{H}}{\otimes}\Zinb$ can also be explained from the point of view of Manin products. For that, we recall the following general result seems to have first been noticed in \cite{MR2337188} and is a simple consequence of the Manin-like adjunction 
 \[
\Hom(\calA\medbullet\calB^!,\calC)\cong\Hom(\calA,\calB\medcirc\calC)
 \]
in the category of quadratic operads generated by binary operations.

\begin{proposition}
For quadratic operads $\calP,\calQ,\calR$ generated by binary operations, there is a canonical morphism
 \[
(\calP\medcirc\calQ)\medbullet\calR\to\calP\medcirc(\calQ\medbullet\calR)
 \]
that is the identity map on generators.
\end{proposition}

\begin{proof}
Indeed, the Manin-like adjunction gives us 
\begin{gather*}
\Hom((\calP\medcirc\calQ)\medbullet\calP^!,\calQ)\cong\Hom(\calP\medcirc\calQ,\calP\medcirc\calQ),\\
\Hom(\calQ,\calR^!\medcirc(\calQ\medbullet\calR))\cong\Hom(\calQ\medbullet\calR,\calQ\medbullet\calR),
\end{gather*}
so there are canonical maps in $\Hom((\calP\medcirc\calQ)\medbullet\calP^!,\calQ)$ and in 
$\Hom(\calQ,\calR^!\medcirc(\calQ\medbullet\calR))$ corresponding to the identity maps. Composing these two maps, we get a canonical 
map in
 \[
\Hom((\calP\medcirc\calQ)\medbullet\calP^!,\calR^!\medcirc(\calQ\medbullet\calR)).
 \]
Additionally, we have
\begin{multline*}
\Hom((\calP\medcirc\calQ)\medbullet\calR,\calP\medcirc(\calQ\medbullet\calR))\cong\\
\Hom((\calP\medcirc\calQ),\calR^!\medcirc\calP\medcirc(\calQ\medbullet\calR))\cong
\Hom((\calP\medcirc\calQ),\calP\medcirc\calR^!\medcirc(\calQ\medbullet\calR))\cong\\
\Hom((\calP\medcirc\calQ)\medbullet\calP^!,\calR^!\medcirc(\calQ\medbullet\calR)),
\end{multline*}
so there is a canonical map in $\Hom((\calP\medcirc\calQ)\medbullet\calR,\calP\medcirc(\calQ\medbullet\calR))$, as required.
\end{proof}

If we apply this result to $\calQ=\Com$, $\calR=\PL$, we get a map 
 \[
(\calP\medcirc\Com)\medbullet\PL\to\calP\medcirc(\Com\medbullet\PL),
 \]
which is simply a map 
 $
\mathrm{DSu}(\calP)\to\calP\medcirc\Zinb.
 $
Recalling that there is a canonical map from the Manin white product maps onto the suboperad of the Hadamard product generated by the binary operations, we obtain the desired map
 \[
\mathrm{DSu}(\calP)\to \calP\underset{\mathrm{H}}{\otimes}\Zinb. 
 \]

An important consequence of Corollary \ref{cor:RBC-P} is the following conceptual explanation of the following general version of the construction of Foissy.

\begin{definition}
Let $A$ be a $\calP$-algebra with a strictly positive weight grading. Let us define, for each generator $\omega$ of $\calP$, the operations 
$\omega_\prec$ and $\omega_\succ$ on $A$ by putting
 \[
\omega_\prec(a_1,a_2)=\frac{|a_1|}{|a_1|+|a_2|}\omega(a_1,a_2),\\
\omega_\succ(a_1,a_2)=\frac{|a_2|}{|a_1|+|a_2|}\omega(a_1,a_2).
 \]
We denote by $\calF(A)$ the vector space $A$ equipped with these operations. 
\end{definition}

The construction $\calF(A)$ is clearly functorial in $A$. We shall refer to it as the \emph{Foissy functor}.

\begin{proposition}
The Foissy functor is a functor from the category of $\calP$-algebras to the category of $\mathrm{DSu}(\calP)$-algebras. 
\end{proposition}

\begin{proof}
Let us consider the $\RB(\Com)$-algebra $\k[z]$ equipped with the obvious commutative algebra structure and the Rota--Baxter operation $R(f):=\int_0^z f(z)\,dz$. Then for a $\calP$-algebra $A$ with a strictly positive weight grading, we can consider the subspace
 \[
\bigoplus_{i>0} A_iz^{i-1}\subset A[z],  
 \]
which, as a vector space, is isomorphic to $A$, is closed under the canonical $\mathrm{DSu}(\calP)$-algebra structure on $A[z]=A\otimes\k[z]$ of Corollary \ref{cor:RBC-P}. This structure is explicitly given by 
 \[
\omega_\prec(a_1,a_2)=\frac{1}{|a_2|}\omega(a_1,a_2),\\
\omega_\succ(a_1,a_2)=\frac{1}{|a_1|}\omega(a_1,a_2),
 \]
and using the same argument as in \cite{MR3818652} (that is, conjugating by the grading operator) furnishes an isomorphism with the structure operations of the Foissy functor.
\end{proof}

\subsection{The operad of rational functions}\label{sec:ratfct}

Recall that the operad $\RF$ of Loday \cite{MR2676956} is defined as follows. We have
 \[
\RF(n)=\k(x_1,\ldots,x_n)
 \]
and the composition rule defined, for $f\in \RF(m)$, $g\in\RF(n)$, and $1\le i\le m$, by
\begin{multline*}
(f\circ_i g)(x_1,\ldots,x_{m+n-1})=\\
f(x_1,\ldots,x_{i-1},x_i+\cdots+x_{i+n-1},x_{i+n},\ldots,x_{m+n-1})g(x_i,\ldots,x_{i+n-1})
\end{multline*}
This makes $\RF$ a nonsymmetric operad; in fact, the composition rule is compatible with the action of symmetric groups, and so one may also view this operad as a symmetric operad. The binary operation $\mu:=1\in\RF(2)$ is immediately seen to generate a suboperad isomorphic to $\Com$. We shall use this latter operation to define an associative algebra structure on $\prod_{n\ge 0}\RF(n)$ as the convolution product on
 \[
\prod_{n\ge 0}\RF(n)\cong\prod_{n\ge 0}\Hom(\As^*(n),\RF(n)).
 \]
In this convolution algebra, the exponential of each element that does not have arity zero component is well defined, as well as the logarithm of elements that have $1$ as the arity zero component. Moreover, it turns out that sending the Rota--Baxter operator to $\frac{1}{x_1}\in\RF(1)$ extends the inclusion $\Com\hookrightarrow\RF$ to an injective map of operads $\RB(\Com)\to \RF$; we present a proof of this fact in Appendix \ref{app:RB}. In particular, the discussion of Rota--Baxter algebras and disuccessors in the previous section immediately implies that the binary operation 
 \[
\nu:=\mu\circ_1R=\frac{1}{x_1}\in \RF(2)
 \] 
satisfies the Zinbiel identity. This result is also an immediate consequence of the corresponding result for the operad $\Mou$ \cite{MR2363304,MR2429249}, since the latter is isomorphic to the operad $\RF$. However, in the context of Rota--Baxter algebras, it is perhaps more practical to use the operad $\RF$, since within that operad the unary operation $x_1$ is a derivation of $\mu$, and so the unary operation $\frac{1}{x_1}$ corresponds to computing anti-derivatives, which matches well Rota's intuition of viewing the Rota--Baxter identity as an algebraic formalization of integration by parts.

\section{Explicit formula for the gauge trivialization}

Now that we offered the reader a suitable recollection of disuccessors and Rota--Baxter operators, we shall use those notions to analyze Equation \eqref{eq:DifferentialOnf-i}, that is
 \[
[f_n,\delta+\alpha] = 
\lambda_n - \sum_{\substack{i+j = n \\ i,j\geq 1}} \frac{1}{j}[f_i,\lambda_j]
 \]
and to deduce from it Property \eqref{eq:GaugeMain}, that is to find an element $\xi$ for which 
 \[
\exp(\ad_\xi)(\delta+\alpha) = \delta+\lambda.\alpha = \sum_{n\ge 1}\frac1n\lambda_n. 
 \]
There are two possible strategies that we are going present. Our first calculation will happen on the level of the Lie algebra of gauge symmetries, and will mainly use calculations in the operad of rational functions (in fact, in its Rota--Baxter suboperad). Our second calculation will happen on the level of the group of gauge symmetries, and will mainly use calculations in free dendriform algebras. We shall conclude by explaining that the seemingly different formulas produced by the two methods actually coincide. 

\subsection{An explicit formula for the gauge symmetry on the Lie algebra level}\label{sec:euler}

Let us explain how to find an element $\xi$ having Property \eqref{eq:GaugeMain}. We shall search for such an element $\xi$ of the form 
\begin{equation}\label{eq:form-xi}
  \xi = \xi_1  + \sum_{n=2}^\infty \xi_n :=
   \sum_{i_1=1}^\infty G_1(i_1) f_{i_1} +  \sum_{n=2}^\infty \sum_{i_1,\dots,i_n \geq 1}
  G_n(i_1,\dots,i_n) \, \ad_{f_{i_n}} \ad_{f_{i_{n-1}}}\cdots \ad_{f_{i_2}} (f_{i_1})\,,
\end{equation}
where $G_n(x_1,\ldots,x_n)$, $n\geq 1$, are unknown rational functions of their arguments.
Let us rewrite Equation \eqref{eq:DifferentialOnf-i} as 
 \[
\frac1n[f_n,\delta+\alpha] = 
\frac1n\lambda_n - \sum_{\substack{i+j = n \\ i,j\geq 1}} \frac{1}{j(i+j)}[f_i,\lambda_j] .
 \]
Applying to this $\frac1{m+n}[f_m,-]$ and summing over all $m$ with constant $m+n$, we obtain 
\begin{multline*}
\sum_{i_1+i_2=n}\frac{1}{i_1(i_1+i_2)}[f_{i_2},[f_{i_1},\delta+\alpha]] = \\
\sum_{i_1+i_2=n}\frac{1}{i_1(i_1+i_2)}[f_{i_2},\lambda_{i_1}] - \sum_{i_1+i_2=n}\sum_{\substack{i+j = i_1 \\ i,j\geq 1}} \frac{1}{j(i+j)(i+j+i_2)}[f_{i_2},[f_i,\lambda_j]]=\\
\sum_{i_1+i_2=n}\frac{1}{i_1(i_1+i_2)}[f_{i_2},\lambda_{i_1}] - \sum_{j_0+j_1+j_2=n} \frac{1}{j_0(j_0+j_1)(j_0+j_1+j_2)}[f_{j_2},[f_{j_1},\lambda_{j_0}]]
 \,,
\end{multline*}
which similarly implies
\begin{multline*}
\sum_{i_1+i_2+i_3=n}\frac{1}{i_1(i_1+i_2)(i_1+i_2+i_3)}[f_{i_3},[f_{i_2},[f_{i_1},\delta+\alpha]]] = \\
\sum_{i_1+i_2+i_3=n}\frac{1}{i_1(i_1+i_2)(i_1+i_2+i_3)}[f_{i_3},[f_{i_2},\lambda_{i_1}]] - \\ \sum_{j_0+j_1+j_2+j_3=n} \frac{1}{j_0(j_0+j_1)(j_0+j_1+j_2)(j_0+j_1+j_2+j_3)}[f_{j_3},[f_{j_2},[f_{j_1},\lambda_{j_0}]]]
 \,,
\end{multline*}
and, more generally,
\begin{multline}\label{eq:telescope}
\sum_{i_1+\cdots+i_k=n}\prod\limits_{j=1}^k\frac{1}{i_{1}+\cdots+i_{j}}\ad_{f_{i_k}}\cdots\ad_{f_{i_1}}(\delta+\alpha) = \\
\sum_{i_1+\cdots+i_k=n}\prod\limits_{j=1}^k\frac{1}{i_{1}+\cdots+i_{j}}\ad_{f_{i_k}}\cdots\ad_{f_{i_2}}(\lambda_{i_1}) - 
\sum_{j_0+\cdots+j_k=n}\prod\limits_{j=0}^k\frac{1}{i_{0}+\cdots+i_{j}}\ad_{f_{i_k}}\cdots\ad_{f_{i_1}}(\lambda_{i_0})
 \,.
\end{multline}
Our computation shows that if the degree $k$ part of $\exp(\ad_{\xi})(\delta+\alpha)$ is given by 
 \[
\sum_{i_1+\cdots+i_k=n}\prod\limits_{j=1}^k\frac{1}{i_{1}+\cdots+i_{j}}\ad_{f_{i_k}}\cdots\ad_{f_{i_1}}(\delta+\alpha)
 \]
for all $k$, then we have a telescoping series made of Equations \eqref{eq:telescope} which adds up to the required 
 \[
\exp(\ad_{\xi})(\delta+\alpha) = \delta+\lambda.\alpha .
 \]
To proceed, we would like to express $\ad_{\xi}$ via the operators $\ad_{f_j}$. To do so, we introduce the set $I_n\subset S_n$ of $V$-shaped permutations, that is permutations $\sigma$ such that 
  \[
\sigma(1) > \sigma (2) >\cdots >\sigma(\sigma^{-1}(1) -1) > \sigma(\sigma^{-1}(1) )=1 < \sigma(\sigma^{-1}(1) +1) < \cdots <\sigma(n)  \,.
  \]

\begin{lemma}\label{lm:preparation-xi-exp}
For an arbitrary $x\in \mathfrak{b}_{-1}$, we have 
  \[
[\xi_n, x] =\sum_{i_1,\dots,i_n \geq 1}
  F_n(i_1,\dots,i_n)\,  \ad_{f_{i_n}} \ad_{f_{i_{n-1}}}\cdots \ad_{f_{i_1}} (x)\,,
  \]
where
\begin{equation} \label{eq:Definition-F-G}
    F_n(i_1,\dots,i_n) \coloneqq \sum_{\sigma\in I_n} (-1)^{\sigma^{-1}(1)-1} G_n \big(i_{\sigma^{-1}(1)},\dots,i_{\sigma^{-1}(n)}\big)
\end{equation}
and if $m\geq 1$, $n_1,\dots,n_m\geq 1$, $\sum_{i=1}^m n_m = n$, we have
  \[
  \ad_{\xi_{n_m}} \cdots \ad_{\xi_{n_1}} (x) =\sum_{i_1,\dots,i_n \geq 1}
  \left[\prod_{j=1}^m F_{n_m} (i_{n_1+\cdots+n_{m-1}+1}, \dots, i_{n_1+\cdots+n_{m}})\right] \ad_{f_{i_n}} \ad_{f_{i_{n-1}}}\cdots \ad_{f_{i_1}} (x)\,.
  \]
\end{lemma}

\begin{proof} 
First, we note that we have
  \[
[\ad_{f_{i_n}} \ad_{f_{i_{n-1}}}\cdots \ad_{f_{i_2}} (f_{i_1}), x] = 
  \sum_{\sigma\in I_n}(-1)^{\sigma^{-1}(1)-1}\ad_{f_{i_\sigma(n)}} \ad_{f_{i_{\sigma(n-1)}}}\cdots \ad_{f_{i_{\sigma(1)}}} (x) \,.
  \]
This is established by induction on $n$. The basis of induction $[f_{i_1},x]=[f_{i_1},x]$ is trivial. Since all $f_i$ are of degree zero, the element $\ad_{f_{i_{k}}}\cdots \ad_{f_{i_2}} (f_{i_1})$ is also of degree zero, and the Jacobi identity implies 
\begin{multline*}
[\ad_{f_{i_n}} \ad_{f_{i_{n-1}}}\cdots \ad_{f_{i_2}} (f_{i_1}),x]=\\
[f_{i_n},[\ad_{f_{i_{n-1}}}\cdots \ad_{f_{i_2}} (f_{i_1}),x]]-
[\ad_{f_{i_{n-1}}}\cdots \ad_{f_{i_2}} (f_{i_1}),[\ad_{f_{i_n}},x]],
\end{multline*} 
and the inductive hypothesis easily completes the proof. The main statement follows automatically from the definition of $\xi_n$.
\end{proof}

Using Lemma~\ref{lm:preparation-xi-exp}, we see that we wish to find rational functions $G_i$, $i\geq 1$, such that the rational functions $F_i$, $i\geq 1$, defined in terms of $G_i$'s in Equation~\eqref{eq:Definition-F-G} satisfy the equations
 \[
\sum_{n_1+\cdots+n_k=n}\prod_{m=1}^k F_{n_m} (i_{n_1+\cdots+n_{m-1}+1}, \dots, i_{n_1+\cdots+n_{m}})=\prod\limits_{j=1}^n\frac{1}{i_{1}+\cdots+i_{j}}.
 \]

Using the convolution algebra constructed out of the operad $\RF$, we can rewrite these equations as
\begin{equation}\label{eq:Exp-F-E}
  \exp\Big[\sum_{j=1}^\infty F_j \Big] (i_1,i_2,\dots) 
  = 1 + \sum_{n=1}^\infty \prod\limits_{j=1}^n\frac{1}{i_{1}+\cdots+i_{j}}. \,,
\end{equation}
or, equivalently, as
\begin{equation}\label{eq:Log-E-F}
  \Big[\sum_{j=1}^\infty F_j \Big] (i_1,i_2,\dots) = \log \left[1+ \sum_{n=1}^\infty \prod\limits_{j=1}^n\frac{1}{i_{1}+\cdots+i_{j}}.\right]\,.
\end{equation}

Recall that a \emph{descent} of a permutation $\rho\in S_n$ is a pair $(i,i+1)$ with $\rho(i) > \rho(i+1)$. We denote by $\cd(\rho)$) the number of  descents of $\rho$.

\begin{proposition}\label{prop:solomon}
We have
 \[
F_n(i_1,\dots,i_n) = \frac1n \sum_{\rho\in S_n} \frac{(-1)^{\cd(\rho^{-1})}}{
  \binom{n-1}{ \cd(\rho^{-1}) }
} 
\prod_{j=1}^{n} \frac {1}{i_{\rho(1)}+\cdots+i_{\rho(j)}}. 
 \]
\end{proposition}

\begin{proof}
We shall instead prove the identity
 \[
\log \left[1+ \sum_{n=1}^\infty \prod\limits_{j=1}^{n-1}\frac{1}{i_{1}+\cdots+i_{j}}.\right]=\sum_{\rho\in S_n} \frac 1n \frac{(-1)^{\cd(\rho^{-1})}}{\binom{n-1}{\cd(\rho^{-1})}} 
\prod_{j=1}^{n-1} \frac {1}{i_{\rho(1)}+\cdots+i_{\rho(j)}}
 \]
which is related to the identity we wish by a conjugation of the commutative product using the Rota--Baxter operator $R$. That conjugation turns the commutative product with respect we compute the logarithm into a commutative product admitting a Zinbiel splitting, which is conceptually important for us. Indeed, the elements 
 \[
\prod_{j=1}^{n-1} \frac {1}{i_{\rho(1)}+\cdots+i_{\rho(j)}}
 \] 
form a basis of the component $\Zinb(n)$ of the Zinbiel operad; in fact, the correspondence is simply
 \[
\prod_{j=1}^{n-1} \frac {1}{i_{\rho(1)}+\cdots+i_{\rho(j)}}\leftrightarrow \nu(\ldots\nu(\nu(a_{\rho(1)},a_{\rho(2)}),\ldots),a_{\rho(n)})
 \]
The $n$-th component of the Zinbiel operad can also be viewed as the multilinear part of the free Zinbiel algebra on $n$ generators, which is the reduced tensor space 
 \[
\overline{T}(V):=\bigoplus_{n\ge 1}V^{\otimes n}. 
 \]
The Zinbiel product is ``one-half'' of the shuffle product on that space; we may also use the deconcatenation coproduct that makes it into a coalgebra. The convolution product between the shuffle product and the deconcatenation coproduct is the dual to the more commonly used convolution between the concatenation product and the shuffle coproduct. For that latter, the logarithm of the identity is precisely the first Eulerian idempotent   
  \[
  \cE\coloneqq \frac1n\sum_{\sigma\in S_n} \frac{(-1)^{\cd(\rho)}}{\binom{n-1}{\cd(\rho)}}  \sigma, 
  \]
see, e.g. \cite{MR1274784}. It remains to note that dualizing amounts, from the point of view of the symmetric group actions (and this is the only thing we need here), to replacing $\sigma$ by $\sigma^{-1}$, and the result follows. 
\end{proof}

\begin{remark}
A different situation where considering the shuffle \emph{product} and the deconcatenation \emph{coproduct} has meaning can be found in the recent work of Foissy and Patras \cite{foissy2024commutativebinftyalgebrasshuffle}. 
\end{remark}

\begin{lemma} \label{lem:Log-General} For
  \[
  G_n(i_1,\dots,i_n)\coloneqq \frac1n\sum_{\substack{\rho\in S_n\\ \rho(1)=1}}  \frac{(-1)^{\cd(\rho^{-1})}}{\binom{n-1}{\cd(\rho^{-1})}} 
\prod_{j=1}^{n} \frac {1}{i_{\rho(1)}+\cdots+i_{\rho(j)}}
  \]
  we have the identity
  \[
  F_n(i_1,\dots,i_n) \coloneqq \sum_{\sigma\in I_n} (-1)^{\sigma^{-1}(1)-1} G_n \big(i_{\sigma^{-1}(1)},\dots,i_{\sigma^{-1}(n)}\big).
  \]
\end{lemma}

\begin{proof} 
Note that 
\begin{multline}\label{eq:solomon-dynkin}
\frac1n\sum_{\rho\in S_n}\frac{(-1)^{\cd(\rho)}}{\binom{n-1}{\cd(\rho)}}a_{\rho(1)}a_{\rho(2)}\cdots a_{\rho(n)}=\\
\frac1n\sum_{\rho\in S_n,\rho(1)=1}\frac{(-1)^{\cd(\rho)}}{\binom{n-1}{\cd(\rho)}} [\ldots[[a_1,a_{\rho(2)}],a_{\rho(3)}],\ldots,a_{\rho(n)}].
\end{multline}
To see that, we note that the left hand side is in the image in the first Solomon idempotent, and hence is known to be a Lie element, so, in particular, it can be expressed as a linear combination of the Dynkin-type nested commutators displayed on the right hand side; comparing coefficients of the noncommutative monomials starting from $a_1$ completes the proof. (This simple but powerful observation is essentially already present in the work of Dynkin \cite[\S 2]{MR30962}, and was rediscovered more recently in \cite{MR4206665,bandiera2017eulerianidempotentprelielogarithm,MR3115178}.) To continue, one expands 
 \[
[\ldots[[a_1,a_{\rho(2)}],a_{\rho(3)}],\ldots,a_{\rho(n)}]
 \]
in the basis of associative noncommutative monomials, which gives us a sum over $I_n$. Finally, remembering, like in the proof of Proposition \ref{prop:solomon}, that the elements 
 \[ 
\prod_{j=1}^{n-1} \frac {1}{i_{\rho(1)}+\cdots+i_{\rho(j)}}
 \]
correspond to 
 \[
\nu(\ldots\nu(\nu(a_{\rho(1)},a_{\rho(2)}),\ldots),a_{\rho(n)}),
 \]
which, in turn, correspond to the elements $a_{\rho(1)}a_{\rho(2)}\cdots a_{\rho(n)}$ in the reduced tensor space, the result follows. 
\end{proof}

\begin{remark} 
The statement of Lemma \ref{lem:Log-General} can also be directly derived from the identity
  \[
  \cE\cdot \cV_i = \cE,
  \] 
  where $\cE$ is the first Eulerian idempotent defined above and 
  \[
  \cV_i \coloneqq (-1)^{i-1}\, \sum_{\substack {\lambda \in I_n \\ \lambda(i) =1 } } \lambda, \qquad i=1,\dots,n.
  \]
  In order to prove this identity, note that $\cE$ and $\cV_i$ belong to the so-called Solomon descent algebra~\cite{Garsia-Reutenauer}. Then one can use the expansion of $\cE$ and $\cV_i$ in the basis elements and the multiplication table in the Solomon descent algebra \cite[Equation~(1.3)]{Garsia-Reutenauer}.
\end{remark}

This lemma, together with our previous results, implies the following statement, which is what we aimed to prove.

\begin{theorem}\label{th:formulaLie} For the element $\xi$ represented as 
  \begin{align*}
    & \xi  = \sum_{n=1}^\infty \sum_{i_1,\dots,i_n \geq 1}
    \Big[\frac1n\sum_{\substack{\rho\in S_n\\ \rho(1)=1}} \frac {(-1)^{\cd(\rho^{-1})}}{\binom{n-1}{\cd(\rho^{-1})}}  
\prod_{j=1}^{n} \frac {1}{i_{\rho(1)}+\cdots+i_{\rho(j)}}\Big] \, \ad_{f_{i_n}} \ad_{f_{i_{n-1}}}\cdots \ad_{f_{i_2}} (f_{i_1})\,,
  \end{align*}
  we have
  \[
  \exp(\ad_{\xi})(\delta+\alpha) = \delta+\lambda.\alpha.
  \]
\end{theorem}

\subsection{An explicit formula for the gauge symmetry on the group level}\label{sec:dendriform}

For this argument, it will be convenient to rewrite our equation in the standard Foissy functor form
\begin{equation}\label{eq:DifferentialOnf-i-shifted}
[b_n,d] = 
a_n - \sum_{\substack{i+j = n \\ i,j\geq 1}} \frac{i}{i+j}[b_i,a_j]\,, 
\end{equation}
where we denote $a_n:=\frac1n\lambda_n$, $b_n:=\frac1n f_n$, and $d:=\delta+\alpha$. Note that at this point we do not need to know the specific form of elements $a_i, b_i, d$, so we may consider our problem in the completion of the free Lie algebra generated by these symbols. We shall examine it in the completed universal enveloping algebra. The gauge group is the group of group-like elements of the universal enveloping algebra, so Property \eqref{eq:GaugeMain} can be written in the form 
\begin{equation}\label{eq:AppndxGaugeMain-shifted-logarithm}
(1+x)d = \left(d+ \sum_{i=1}^\infty a_i\right)(1+x) \,.
\end{equation}
Recall the weight grading $w$ introduced above; it defines a weight grading on the universal enveloping algebra for which we have $w(a_i)=w(b_i)=i$, and $w(d)=0$. Let us consider the ideal of the universal enveloping algebra which is the kernel of the map to the universal enveloping of the Lie algebra generated by $d$ alone. The grading on that ideal induced by the weight grading $w$ is strictly positive, so we can use the Foissy functor to convert it into a dendriform algebra. Moreover, we can adjoint $d$ to that dendriform algebra provided that we never have to compute dendriform products of two homogeneous elements of degree zero (this generalises the unit trick of \cite[Sec.~2.1]{MR1879927}). Once that is done, we observe that Equation \eqref{eq:DifferentialOnf-i-shifted} has a simple dendriform expression. Indeed, we have 
\begin{gather*}
[b_i,d]=b_id-db_i=b_i\prec d-d\succ b_i\, ,\\
a_i-\sum_{\substack{j+k = i \\ j,k\geq 1}} \frac{j}{i} [b_j,a_k]=
a_i-\sum_{\substack{j+k = i \\ j,k\geq 1}} \frac{j}{i} (b_ja_k-a_kb_j)=a_i-\sum_{\substack{j+k = i \\ j,k\geq 1}} (b_j\prec a_k-a_k\succ b_j) \, ,
\end{gather*}
and therefore if we introduce elements $\alpha=\sum_{i\ge 1}a_i$ and $\beta=\sum_{i\ge 1}b_i$, the  Equation \eqref{eq:DifferentialOnf-i-shifted}
can be rewritten as 
\begin{equation}\label{eq:Dendriform}
\beta\prec (d+\alpha)=\alpha+(d+\alpha)\succ\beta ,
\end{equation}
and Equation \eqref{eq:AppndxGaugeMain-shifted-logarithm} becomes
 \[
(1+x)d=(d+ \alpha)(1+x).
 \]
Our next result furnishes such element; the computation below is closely related to recent work \cite{MR2556138,MR2422313} on equations and identities in dendriform algebras, see also \cite{MR3070646}. To state the following result, we introduce right-normed dendriform products by putting, for an element $u$ of a dendriform algebra, $u^{\prec 1}=u$ and $u^{\prec (n+1)}=u\prec u^{\prec n}$.

\begin{theorem}\label{th:formulaGp}
The element 
 \[
1+\sum_{k\ge 1} \beta^{\prec k}
 \]
is group-like for the standard coproduct $\Delta$ of the universal enveloping algebra and satisfies 
 \[
\left(1+\sum_{k\ge 1}\beta^{\prec k}\right)d=(d+ \alpha)\left(1+\sum_{k\ge 1}\beta^{\prec k}\right).
 \]
\end{theorem}

\begin{proof}
Let us first prove the group-like property. We write $\beta^{\prec m}=\sum_{i\ge 1} b_i^{(m)}$, where $w(b_i^{(m)})=i$, and impose $b_i^{(0)}=\delta_{i,0}$. It is enough to show that for all $i,k$ we have
 \[
\Delta(b_i^{(k)})=\sum_{p=0}^{k}\sum_{r=0}^{i}b_r^{(p)}\otimes b_{i-r}^{(k-p)} \, .
 \]
We shall prove this result by induction on $k$. For $k=1$, this holds because $\beta=\sum_{i\ge 1}b_i$ is a primitive element of $U$ by definition. Suppose that we established this formula for $k=m$, and let us prove it for $k=m+1$. 
We have the following equalities 
\begin{multline*}
\Delta(b_s^{(m+1)})=\Delta\left(\sum_{i=1}^{s-1} b_i\prec b_{s-i}^{(m)}\right)
=\sum_{i=1}^{s-1}\Delta\left(b_i\prec b_{s-i}^{(m)}\right)=\sum_{i=1}^{s-1}\frac{i}{i+j}\Delta\left(b_ib_{s-i}^{(m)}\right)=\\
\sum_{i=1}^{s-1}\frac{i}{s}(b_i\otimes 1+1\otimes b_i)\left(\sum_{p=0}^{m}\sum_{r=0}^{s-i}b_r^{(p)}\otimes b_{s-i-r}^{(m-p)}\right)
=\sum_{i=1}^{s-1}\frac{i}{s}\sum_{p=0}^{m}\sum_{r=0}^{s-i}\left(b_ib_r^{(p)}\otimes b_{s-i-r}^{(m-p)}+b_r^{(p)}\otimes b_ib_{s-i-r}^{(m-p)}\right)\\
=\sum_{i=1}^{s-1}\sum_{p=0}^{m}\sum_{r=0}^{s-i}\left(\frac{i+r}{s}b_i\prec b_r^{(p)}\otimes b_{s-i-r}^{(m-p)}+\frac{s-r}{s}b_r^{(p)}\otimes b_i\prec b_{s-i-r}^{(m-p)}\right) .
\end{multline*}
We collect all elements with the fixed value of $i+r$ in the first sum, and with the fixed value of $s-r$ in the second sum, obtaining
 \[
\Delta(b_s^{(m+1)})=
\sum_{u=1}^{s}\sum_{p=0}^{m}\frac{u}{s}b_u^{(p+1)}\otimes b_{s-u}^{(m-p)}+
\sum_{r=0}^{s}\sum_{p=0}^{m}\frac{s-r}{s}b_r^{(p)}\otimes b_{s-i}^{(m-p)}=
\sum_{a=0}^{s}\sum_{p=0}^{m+1}b_a^{(p)}\otimes b_{s-a}^{(m+1-p)} \, ,
 \]
as required.

To complete the proof, we shall perform some further dendriform calculations. Using the properties $d\prec u=u\succ d=0$ for $u$ of positive weight grading and the dendriform axioms, we obtain
\begin{multline*}
\left(1+\sum_{k\ge 1} \beta^{\prec k}\right)d=d+\beta d+\sum_{k\ge 1} \beta^{\prec (k+1)}d
=d+\beta\prec d+\sum_{k\ge 1}(\beta\prec \beta^{\prec k})\prec d=\\
d+\beta\prec d+\sum_{k\ge 1}\beta\prec(\beta^{\prec k}d)
=d+\beta\prec d+\sum_{k\ge 1}\beta\prec(\beta^{\prec k} d) 
=d+\beta\prec \left[\left(1+\sum_{k\ge 1}\beta^{\prec k}\right) d\right] 
\end{multline*}
Similarly, using those properties and Equation \eqref{eq:Dendriform}, we obtain
\begin{multline*}
(d+\alpha)\left(1+\sum_{k\ge 1} \beta^{\prec k}\right)=d+\alpha+(d+\alpha)\beta+\sum_{k\ge 1}(d+\alpha)\beta^{\prec (k+1)}=\\
d+\alpha+(d+\alpha)\prec \beta+(d+\alpha)\succ \beta+\sum_{k\ge 1}((d+\alpha)\prec \beta^{\prec (k+1)} + (d+\alpha)\succ (\beta\prec \beta^{\prec k}))=\\
d+\alpha+\alpha\prec \beta+(d+\alpha)\succ \beta+\sum_{k\ge 1}(\alpha\prec \beta^{\prec (k+1)} + ((d+\alpha)\succ \beta)\prec \beta^{\prec k})=\\
d+\alpha+\alpha\prec \beta+\beta\prec (d+\alpha)-\alpha+\sum_{k\ge 1}(\alpha\prec \beta^{\prec (k+1)}+(\beta\prec (d+\alpha)-\alpha)\prec \beta^{\prec k})=\\
d+\alpha\prec \beta+\beta\prec \alpha+\sum_{k\ge 1}(\alpha\prec \beta^{\prec (k+1)}-\sum_{k\ge 1}\alpha\prec \beta^{\prec k}+\sum_{k\ge 1}(\beta\prec (d+\alpha))\prec \beta^{\prec k})=\\
d+\alpha\prec \beta+\beta\prec \alpha-\alpha\prec \beta+\sum_{k\ge 1}\beta\prec ((d+\alpha) \beta^{\prec k})=\\
d+\beta\prec (d+\alpha)+\sum_{k\ge 1}\beta\prec ((d+\alpha) \beta^{\prec k})
=d+\beta\prec \left[(d+\alpha)\left(1+\sum_{k\ge 1}\beta^{\prec k}\right)\right]
\end{multline*}
Therefore, both sides of the equality 
 \[
\left(1+\sum_{k\ge 1} \beta^{\prec k}\right)d=(d+\alpha)\left(1+\sum_{k\ge 1} \beta^{\prec k}\right)
 \]
are solutions to the equation $F=d+\beta\prec F$ which clearly determines the element $F$ uniquely.
\end{proof}

\subsection{Comparison of two formulas}

Let us explain how to compare the two formulas. The resulting formula of Theorem \ref{th:formulaLie} can be rewritten as 
 \begin{align*}
    \xi & = \sum_{n=1}^\infty \sum_{i_1,\dots,i_n \geq 1}
    \Big[\frac1n\sum_{\substack{\rho\in S_n\\ \rho(1)=1}} \frac {(-1)^{\cd(\rho^{-1})}}{\binom{n-1}{\cd(\rho^{-1})}}  
\prod_{j=1}^{n} \frac {1}{i_{\rho(1)}+\cdots+i_{\rho(j)}}\Big] \, \ad_{f_{i_n}} \cdots \ad_{f_{i_2}} (f_{i_1})\\
&= \sum_{n=1}^\infty \sum_{\substack{i_1,\dots,i_n \geq 1\\ \rho\in S_n\colon \rho(1)=1}}
    \Big[\frac1n \frac {(-1)^{\cd(\rho^{-1})}}{\binom{n-1}{\cd(\rho^{-1})}}  
\prod_{j=1}^{n} \frac {1}{i_{\rho(1)}+\cdots+i_{\rho(j)}} \, \ad_{f_{i_{\rho^{-1}(\rho(n))}}} \cdots \ad_{f_{i_{\rho^{-1}(\rho(2))}}} (f_{i_{\rho^{-1}(\rho(1))}})\Big]\\
&= \sum_{n=1}^\infty  \sum_{\substack{i_1,\dots,i_n \geq 1\\ \rho\in S_n\colon \rho(1)=1}}
    \Big[\frac1n\frac {(-1)^{\cd(\rho^{-1})}}{\binom{n-1}{\cd(\rho^{-1})}}  
\prod_{j=1}^{n} \frac {1}{i_{1}+\cdots+i_{j}} \, \ad_{f_{i_{\rho^{-1}(n)}}} \cdots \ad_{f_{i_{\rho^{-1}(2)}}} (f_{i_{\rho^{-1}(1)}})\Big]\\
&=\sum_{n=1}^\infty \sum_{i_1,\dots,i_n \geq 1}\prod_{j=1}^{n} \frac {1}{i_{1}+\cdots+i_{j}}
\sum_{\substack{\rho\in S_n\\ \rho(1)=1}}
    \Big[\frac1n\frac {(-1)^{\cd(\rho^{-1})}}{\binom{n-1}{\cd(\rho^{-1})}}  
 \, \ad_{f_{i_{\rho^{-1}(n)}}} \ad_{f_{i_{\rho^{-1}(n-1)}}}\cdots \ad_{f_{i_{\rho^{-1}(2)}}} (f_{i_{\rho^{-1}(1)}})\Big]\\
&= \sum_{n=1}^\infty \sum_{i_1,\dots,i_n \geq 1}\prod_{j=1}^{n} \frac {1}{i_{1}+\cdots+i_{j}}
\sum_{\substack{\rho\in S_n\\ \rho(1)=1}}
    \Big[\frac1n\frac {(-1)^{\cd(\rho^{-1})}}{\binom{n-1}{\cd(\rho^{-1})}}  f_{i_{\rho^{-1}(n)}}f_{i_{\rho^{-1}(n-1)}}\cdots f_{i_{\rho^{-1}(2)}}f_{i_{\rho^{-1}(1)}}\Big]
\end{align*}

 \[
\sum_{n=1}^\infty \sum_{i_1,\dots,i_n \geq 1} \cE\Big(\prod_{j=1}^{n}\frac {1}{i_{1}+\cdots+i_{j}}
 f_{i_{n}}f_{i_{n-1}}\cdots f_{i_{2}}f_{i_{1}}\Big),
 \]
or, recalling the notation of Section \ref{sec:dendriform},
\begin{multline*}
\sum_{n=1}^\infty \sum_{i_1,\dots,i_n \geq 1} \cE\Big(\prod_{j=1}^{n}\frac {i_j}{i_{1}+\cdots+i_{j}}
 b_{i_{n}}b_{i_{n-1}}\cdots b_{i_{2}}b_{i_{1}}\Big)=\\
\sum_{n=1}^\infty \sum_{i_1,\dots,i_n \geq 1} \cE\Big(
 b_{i_{n}}\prec(b_{i_{n-1}}\prec (\cdots (b_{i_{2}}\prec f_{b_{1}})\cdots))\Big),
\end{multline*}
where $\prec$ is the dendriform operation obtained by applying the Foissy functor. This is in perfect agreement with the result of Theorem \ref{th:formulaGp}: we apply the first Eulerian idempotents to a group-like element, thus extracting its logarithm.

\appendix

\section{The rational functions representation of the Rota--Baxter identity}\label{app:RB}

In this section, we prove the claim of Section \ref{sec:ratfct} that the morphism from the operad $\RB=\RB(\Com)$ of commutative Rota--Baxter algebras to the operad $\RF$ is injective. Let us start with a combinatorial description of a basis of the operad $\RB$. For that, we introduce the following combinatorial construction. For a positive integer $n$, to a sequence of nested sets 
 \[
\varnothing\neq I_1\subsetneq I_2\subsetneq \ldots \subsetneq I_k\subseteq\{1,\ldots,n\}
 \]
and a sequence of positive integers $(d_1,\ldots,d_k)$, we may associate a \emph{nested Rota--Baxter monomial}
 \[
R^{d_k}(\cdots (R^{d_2}(R^{d_1}(a_{I_1})a_{I_2\setminus I_1}))\cdots a_{I_k\setminus I_{k-1}})a_{\{1,\ldots,n\}\setminus I_k},
 \]
where we put $a_J:=\prod_{i\in J} a_i$. The following result can be inferred from the construction of the free Rota--Baxter algebra \cite{MR1744484}, but we give an alternative argument using operad theory, which we believe to be of independent interest.

\begin{proposition}\label{prop:RB-monomials}
The nested Rota--Baxter monomials form a basis in the component $\RB(n)$ of the Rota--Baxter operad.
\end{proposition}

\begin{proof}
First, the Rota--Baxter relation may be viewed as a rewriting rule that replaces $R(f)R(g)$, an element where $R$ is applied before computing the product by a combination of terms where one copy of $R$ is applied after computing the product. Using that rewriting rule, it is easy to prove by an inductive argument that nested Rota--Baxter monomials span the $n$-th component of the operad $\RB$. To show that they are linearly independent, we shall use a dimension counting argument. Components of the operad $\RB$ are infinite-dimensional, so one has to be a bit careful. Our strategy is to introduce the weight grading for which $R$ is of weight one and $\mu$ is of weight zero; for that weight grading the Rota--Baxter relation is homogeneous of degree two, and all the weight graded components of the operad $\RB$ are finite-dimensional. To show that the nested Rota--Baxter monomials form a basis, it is enough to show that for each weight grading $d$, the number of nested Rota--Baxter monomials with $d_1+\cdots+d_k=d$ is equal to the dimension of the respective weight graded component of $\RB(n)$. To compute that latter dimension independently, we may use the result of \cite[Cor.~5.10]{MR3084563} that describes the minimal model of the operad $\RB$. Indeed, it is known that, if we denote by $P(t)$ the Poincaré series of the space of generators of the minimal model of an operad, the series $t-P(t)$ is the compositional inverse of the Poincaré series of an operad; an extra weight grading can be incorporated in that statement without any problem. The minimal model of \emph{loc. cit.} has, in each arity $n$, $(n-1)!$ basis elements of homological degree $n-2$ and weight zero and $(n-1)!$ basis elements of homological degree $n-1$ and weight $n$. Therefore, the compositional inverse of the Poincaré series of the operad $\RB$ is
 \[
\sum_{n\ge 1}\frac{(-1)^{n-1}t^n(1-q^n)}{n}.
 \]  
It turns out that the inverse of that series is given by $\sum_{n\ge 1}\frac{P_n(q)}{n!(1-q)^n}$, where $P_n(q)$ is the Eulerian polynomial (the polynomial enumerating permutations according to their number of descents); this is essentially \cite[Exercise 3.81]{MR2868112}. It remains to note that the expression
 \[
\frac{P_n(q)}{(1-q)^n} 
 \]
is precisely the Hilbert series of the Stanley--Reisner ring (see e.g. \cite{MR725505}) of the Boolean lattice $2^{\{1,\ldots,n\}}$  which has a basis indexed by the same data (nested sets with multiplicities). 
\end{proof}

We are now ready to prove injectivity of the morphism of operads $\RB\to\RF$. In fact, it costs nothing to prove a slightly more general result. Recall that one can consider, for any given $\theta\in\k$, \emph{Rota--Baxter algebras of weight $\theta$} which are algebras equipped with an associative commutative product and an endomorphism $R$ satisfying the identity
 \[
R(a)R(b)=R(R(a)b+aR(b)+\theta ab).
 \]
For $\theta=0$, one recovers the usual Rota--Baxter algebras, and for $\theta=1$ this is the summation by parts identity 
 \[
S(a)S(b)=S(S(a)b+aS(b)+ab)
 \]
for the operator on sequences 
 \[
S\colon (a_n)_{n\in\mathbb{N}}\mapsto \left(\sum_{k=1}^n a_k\right)_{n\in\mathbb{N}}.
 \] 
We shall denote the corresponding operad by $\RB^\theta$. Note that Loday defined in \cite{MR2676956} an operad $\calO^F$ for any one-dimensional formal group law $F$, that is, a formal power series $F(x_1,x_2)\in\k[[x_1,x_2]]$ such that 
\begin{gather*}
F(x_1,x_2)-x_1-x_2\in\mathfrak{m}^2,\\
F(F(x_1,x_2),x_3)=F(x_1,F(x_2,x_3)),
\end{gather*}
where $\mathfrak{m}=(x_1,x_2)$ is the maximal ideal of $\k[[x_1,x_2]]$. Namely, one can put $\calO^F(n)=\k[[x_1,\ldots,x_n]]$ and let
 \[
(f\circ_i g)(x_1,\ldots,x_{m+n-1})=
f(x_1,\ldots,x_{i-1},F_n(x_i,\ldots,x_{i+n-1}),x_{i+n},\ldots,x_{m+n-1})g(x_i,\ldots,x_{i+n-1}),
 \]
where $F_n$ is the $n$-ary operation $F(F(\ldots F(F(x_1,x_2)x_3),\ldots),x_n)$. The binary operation $\mu:=1\in\k[[x_1,x_2]]$ always generates a commutative suboperad of $\calO^F$. If, as in the case $F(x_1,x_2)=x_1+x_2$, the formal group law is given by a polynomial (or by a rational function), the composition rule defines an operad structure on the level of polynomials (or on rational functions), i.e. one can replace $\k[[x_1,\ldots,x_n]]$ with $\k[x_1,\ldots,x_n]$ (or with $\k(x_1,\ldots,x_n)$). In particular, for the formal group law 
 \[
F^\theta(x_1,x_2)=x_1+x_2+\theta x_1x_2,
 \]
we may define an operad structure on the level of rational functions, which we denote $\RF^\theta$. For $\theta=0$, Loday proved that this operad is isomorphic to the operad $\Mou$ introduced by Chapoton \cite{MR2363304} who in particular discussed the relationship between this operad and the operad of dendriform algebras. For $\theta=1$, this operad was considered in \cite{MR3537825,MR3071832}, and its relationship with tridendriform algebras was discussed. In the view of \cite{MR2357335}, one may view this as a hint at some underlying Rota--Baxter context. We shall make this precise as follows.

\begin{theorem}
The operations 
 \[
\mu:=1\in\RF^\theta(2) \text{ and } R:=\frac1{x_1}\in\RF^\theta(1)
 \] 
satisfy the identities of Rota--Baxter algebras of weight $\theta$. Moreover, they do not satisfy any further identities, and thus generate a suboperad isomorphic to $\RB^\theta$.
\end{theorem}

\begin{proof}
Note that we have the following compositions in the operad $\RF^\theta$:
\begin{gather*}
R(R(a_1)a_2)=\frac{1}{x_1+x_2+\theta x_1x_2}\cdot\frac1{x_1},\quad 
R(a_1R(a_2))=\frac{1}{x_1+x_2+\theta x_1x_2}\cdot\frac1{x_2},\\
R(a_1a_2)=\frac{1}{x_1+x_2+\theta x_1x_2},\quad
R(a_1)R(a_2)=\frac{1}{x_1}\cdot\frac{1}{x_2},
\end{gather*}
so the identity $R(a_1)R(a_2)=R(R(a_1)a_2+a_1R(a_2)+\theta a_1a_2)$
does indeed hold. Thus, there is a morphism of operads $\RB^\theta\to \RF^\theta$. Let us show that this morphism is injective. In fact, everything here is defined over the ring $\k[\theta]$, and we shall furnish an argument over this ring. Note that the Gr\"obner basis argument of \cite[Prop.~5.6]{MR3084563} essentially shows that the component $\RB^\theta(n)$ is a free $\k[\theta]$-module; moreover, the argument of Proposition~\ref{prop:RB-monomials} can be easily adapted to show that nested Rota--Baxter monomials form a basis of that module. To complete the proof, it remains to show that the images of the nested Rota--Baxter monomials
 \[
R^{d_k}(\cdots (R^{d_2}(R^{d_1}(a_{I_1})a_{I_2\setminus I_1}))\cdots a_{I_k\setminus I_{k-1}})a_{\{1,\ldots,n\}\setminus I_k},
 \]
that is, the rational functions
 \[
\frac1{(F^\theta_{I_1})^{d_1}(F^\theta_{I_2})^{d_2}\cdots (F^\theta_{I_k})^{d_k}} ,
 \]
where for $J=\{j_1,\ldots,j_s\}$ we put 
 \[
F^\theta_{J}:=F^\theta(F^\theta(\ldots F^\theta(F^\theta(x_{j_1},x_{j_2}),x_{j_3}),\ldots),x_{j_s}),
 \] 
are linearly independent. It is enough to show that their images under the specialization $\theta\mapsto 0$ are linearly independent. Suppose the contrary, so that there is a linear combination of these rational functions that is equal to zero. Let us consider the minimal subsets $\{i_1,\ldots,i_{n_1}\}$ of the chains of nested sets that appear in such a linear combination, and let us choose among those minimal subsets that is additionally minimal with respect to inclusion. Choosing a system of coordinates that has $x_{i_1}+\cdots +x_{i_{n_1}}$ as one of the new coordinates and expanding our rational function as a Laurent series in that coordinate, one easily sees by induction that all the coefficients of our linear combination must be equal to zero. 
\end{proof}

\printbibliography
\end{document}